\documentclass{siamltex}

\usepackage{amsfonts,amsmath,color,verbatim}
\usepackage{stmaryrd}
\usepackage[mathscr]{eucal}
\usepackage{epsfig}
\usepackage{comment}
\usepackage[ruled,vlined]{algorithm2e}
\usepackage{color}

\def\eps{\varepsilon}
\def\e          {\ensuremath{\mathrm{e}}}

\def\C          {\ensuremath{\mathbb C}}

\def\R          {\ensuremath{\mathbb R}}

\newcommand{\M}{\mathcal{M}}
\newcommand{\cM}{\mathcal{M}}
\def\bigo{{\mathcal O}}

\def\iu{\mathrm{i}}
\newcommand{\cS}{\mathcal{S}}
\newcommand{\cP}{\mathcal{P}}
\renewcommand{\Re}{{\hbox{\rm Re}\,}}
\renewcommand{\Im}{{\hbox{\rm Im}\,}}
\renewcommand{\e}{{\rm e}}
\newcommand{\clambda}{\overline{\lambda}}

\def\conj{\overline}
\def\clambda{{\overline{\lambda}}}

\def\cz{{\overline{z}}}
\newcommand{\qed}{\hfill \ensuremath{\Box}}
\renewcommand{\Re}{\text{\rm Re}}
\newcommand{\du}{\dot{u}}
\newcommand{\dv}{\dot{v}}
\newcommand{\dx}{\dot{x}}
\newcommand{\dy}{\dot{y}}
\def\oeps{\eps_\star}
\def\ophi{\phi}

\newtheorem{example}[theorem]{Example}
\newtheorem{remark}[theorem]{Remark}
\newtheorem{assumption}[theorem]{Assumption}

\title{Rank-$1$ matrix differential equations for \\
structured eigenvalue optimization}

\author{Nicola Guglielmi\footnotemark[1] \and Christian Lubich\footnotemark[3] \and Stefano Sicilia\footnotemark[1]}
\begin{document}

\maketitle

\renewcommand{\thefootnote}{\fnsymbol{footnote}}
\footnotetext[1]{Division of Mathematics, 
Gran Sasso Science Institute,
Via Crispi 7,
I-67100   L' Aquila,  Italy. Email: {\tt nicola.guglielmi@gssi.it, stefano.sicilia@gssi.it}}
\footnotetext[3]{Mathematisches Institut,
       Universit\"at T\"ubingen,
       Auf der Morgenstelle 10,
       D--72076 T\"ubingen,
       Germany. Email: {\tt lubich@na.uni-tuebingen.de}}
\renewcommand{\thefootnote}{\arabic{footnote}}

\begin{abstract}
A new approach to solving eigenvalue optimization problems for large structured matrices is proposed and studied. The class of optimization problems considered is related to computing structured pseudospectra and their extremal points, and to structured matrix nearness problems such as computing the structured distance to instability or to singularity. The structure can be a general linear structure and includes, for example, large matrices with a given sparsity pattern, matrices with given range and co-range, and Hamiltonian matrices.
Remarkably, the eigenvalue optimization can be performed on the manifold of complex (or real) rank-1 matrices, which yields
a significant reduction of storage and in some cases of the computational cost.
The method relies on a constrained gradient system and the projection of the gradient
onto the tangent space of the manifold of complex rank-$1$ matrices. It is shown that near a local minimizer this projection is very close to the identity map, and so the computationally favorable rank-1 projected system behaves locally like the  
gradient system.

\end{abstract}



\pagestyle{myheadings} \thispagestyle{plain}
\markboth{N. GUGLIELMI,  CH. LUBICH,  S. SICILIA}{Rank-$1$ differential equations for structured eigenvalue optimization}

\section{Introduction}
\label{sec:intro}

We describe an approach to structured eigenvalue optimization problems that uses constrained gradient flows 
and the underlying rank-$1$ property of the optimizers. We illustrate basic techniques on a class 
of model problems, which arise in computing 
structured pseudospectra or their extremal points and appear as the essential algorithmic building block in structured matrix nearness problems.
For example, we determine the largest possible spectral abscissa or radius of a given matrix under perturbations of a prescribed norm that preserve its
structure, or - in other words - the structured pseudospectral abscissa or radius. This is an important subtask in the computation of structured stability radii (or structured distance to instability in another terminology).
In the literature these quantities are extensively studied with the purpose of analyzing stability
properties and robustness of linear dynamical systems (see, e.g., \cite{HinP05}).
Similarly, if one is interested in the distance of a matrix to singularity, the unstructured distance is the smallest singular
value. However, if the matrix is structured, having a small singular value does not imply the existence of a small structured
perturbation that makes it singular, and the structured distance to singularity is not readily obtained.

The structures considered here are general complex- or real-linear structures, including for example matrices with a given sparsity pattern, symmetric such matrices, matrices with prescribed range and co-range, Hamiltonian, Toeplitz and Hankel matrices, and block matrices whose blocks may have any of those properties.

The method we present relies first on a norm- and structure-constrained gradient system and then on its reduction to the manifold of rank-$1$ matrices. Instead of a direct discrete approach to solve the optimization problems, we present a continuous ODE-based optimization method which is crucial to reveal the underlying rank-$1$ property of optimizers, on which we build our method. This property is well-known for unstructured problems 
(see e.g. \cite{TrE05}) and has also been exploited for developing suitable algorithms (see e.g. \cite{GO11,KV14,GL11}).

There are several situations, addressed in the literature, where considering a time-continuous algorithm provides new insight. 
The reader is referred for example to \cite{Bro91,Chu08,HeM94,TrE05,AbMS09} and the references therein. This list is far from being exhaustive.

In previous works, structured eigenvalue optimization problems were addressed for some specific structures. For example when the matrices are required to be real (the unstructured problem would consider them as complex), it has been proved
that the optimizers have a rank-$2$ structure \cite{QiuBRDYD95} and indeed are obtained as real parts of an underlying rank-$1$ matrix 
\cite{GL13}. Similarly, Hamiltonian eigenvalue optimization has been studied in detail 
in \cite{MeX08} 
and \cite{AlaBKMM11}, in the ambit of robust passivity analysis of linear control systems, where eigenvalues of Hamiltonian matrices have to be bounded away from the imaginary axis. In that case it is possible to show that for a real Hamiltonian matrix, extremal perturbations have
rank $4$ \cite{GKL15}. However, when considering for example a sparse matrix, the low-rank property of optimizers seems to be irremediably lost.
It is a basic goal of this article to uncover the underlying rank-$1$ property and to explain how it can be  used in algorithms for structured eigenvalue optimization.

The paper is organized as follows. In Section \ref{sec:problem} we set up the framework and present
our approach, which is based on a structure- and norm-constrained gradient system.
We show that optimizers are orthogonal projections of rank-1 matrices onto the given structure. In Section~\ref{sec:rank1} we introduce a differential equation on the manifold of rank-1 matrices of unit Frobenius norm, for which the stationary points are in a bijective correspondence with the stationary points of the structure- and norm-constrained gradient system.  In Section \ref{sec:locconv} we prove local convergence to strong minima. 
In Section \ref{sec:proto-num} we discretize the rank-1 differential equation by a splitting method. This leads us to a fully discrete algorithm that updates rank-1 matrices in every step. 
Then, in Section \ref{sec:stabrad} we describe a two-level approach to compute structured stability radii (or structured distance to instability), used to characterize robustness of spectral stability properties. This is an important application of the considered
class of eigenvalue optimization problems to solving structured matrix nearness problems. The structured distance to singularity is computed in an analogous way.
Finally, in Section \ref{sec:ill} we present some illustrative examples showing that the rank-1 system
is well-suited for the efficient computation of optimizers.

\section{Problem description}
\label{sec:problem}


Let $A \in \C^{n,n}$ be a given matrix and let $\lambda(A)\in\C$ be a target
eigenvalue of~$A$, for example:
\begin{itemize}
\item[(i) ] \quad the eigenvalue of minimal or maximal real part;
\item[(ii) ] \quad the eigenvalue of minimal or maximal modulus;
\item[(iii) ] \quad the closest eigenvalue to a given set in the complex plane.
\end{itemize}


Let $\cS$ be a subspace of the vector space of complex or real $n\times n$ matrices, e.g. a space of matrices with a prescribed sparsity pattern,  or matrices with given range and co-range, or
 Toeplitz matrices, or Hankel matrices, or Hamiltonian matrices, etc.  

We let
\begin{equation}
\label{ass:f}
f: \C^2 \rightarrow \C \quad\text{with}\quad f\left( z, \cz \right) = f\left( \cz, z \right) \in \R \quad\text{for all }\,z\in\C
\end{equation}
be a given smooth function, 
e.g., $f$ or $-f$ evaluated at $\left( z, \cz \right)$ equals
\[
\Re\,z = \frac{z + \cz}{2}, \quad | z |^2 = z \cz.
\]  

\begin{remark} \rm
An extension to functions of several target eigenvalues is direct, but is not considered in this paper,
for sake of conciseness.
\end{remark}

We consider the following structured eigenvalue optimization problem: For a given perturbation size $\varepsilon>0$, find 
\begin{equation} \label{eq:optimiz0}
\arg\min\limits_{\Delta \in \cS,\, \| \Delta \|_F = \eps} f \left( \lambda\left( A + \Delta \right), \clambda \left( A + \Delta \right)  \right),
\end{equation}
where $\| \Delta \|_F$ is the Frobenius norm of the structured matrix $\Delta\in\cS$, i.e.~the Euclidean norm of the vector of its entries; where
$\lambda(A+\Delta)$ is the considered target eigenvalue of the perturbed matrix $A + \Delta$.
The $\arg\max$ case  is treated analogously, replacing $f$ by $-f$.
%
It is convenient to write
\[
\Delta = \eps E \quad \mbox{ with} \ \| E \|_F = 1
\]
and define
\begin{equation} \label{eq:optimiz0S}
F_\eps(E) = f \left( \lambda\left( A + \eps E \right), \clambda\left( A + \eps E \right)  \right)
\end{equation}
so that Problem~\eqref{eq:optimiz0} is equivalent to the  structured eigenvalue optimization problem 
\begin{equation} \label{eq:optimizS}
\arg\min\limits_{E \in \cS, \| E \|_F = 1} F_\eps(E).
\end{equation}
Problem \eqref{eq:optimiz0} or \eqref{eq:optimizS} is a nonconvex, nonsmooth optimization problem.

In a variant to the above problem, the inequality constraints $\| \Delta \|_F \le \eps$ and $\| E \|_F \le 1$ can also be considered in
\eqref{eq:optimiz0} and~\eqref{eq:optimizS}, respectively.


\subsection{Projection onto the structure}\label{subsec:proj-structure}
In order to treat the above  problem, we shall make use of a projection onto the structure space $\cS$.

Given two complex $n \times n$ matrices, we denote by (tr$(\cdot)$ denotes the trace) 
\begin{equation*}
\langle X,Y \rangle =\sum_{i,j} \conj{x}_{ij}y_{ij} = {\rm tr}(X^* Y)
\end{equation*} 
the inner product in $\C^{n,n}$ that induces the Frobenius norm $\| X \|_F = \langle X,X \rangle^{1/2}$.

Let $\Pi^\cS$ be the orthogonal projection (w.r.t. the Frobenius inner product)  onto~$\cS$\/: for every $Z\in \C^{n,n}$, 
\begin{equation}\label{Pi-S}
\Pi^\cS Z \in \cS \quad\text{ and } \quad \Re\langle \Pi^\cS Z, W \rangle = \Re\langle Z,W \rangle \quad \forall \, W\!\in\cS.
\end{equation}
For a complex-linear subspace $\cS$, taking the real part of the complex inner product can be omitted (because with $W\in\cS$, then also $\iu W\in\cS$), but taking the real part is needed for real-linear subspaces. Note that for $\cS=\R^{n,n}$,  we then have $\Pi^\cS Z=\Re\, Z$ for all $Z\in\C^{n,n}$.
In the following examples, the stated action of $\Pi^\cS$ is readily verified. 

\begin{example}[Sparse matrices] \rm If $\cS$ is the space of complex matrices with a prescribed sparsity pattern, then $\Pi^\cS Z$ leaves the entries of $Z$ on the sparsity pattern unchanged and annihilates those outside the sparsity pattern.  If $\cS$ is the space of real matrices with a prescribed sparsity pattern, then $\Pi^\cS Z$ takes instead the real part of the entries of $Z$ on the sparsity pattern.
\end{example}

\begin{example} [Matrices with prescribed range and co-range]
\rm An example of particular interest in control theory is the perturbation space
\[
\cS = \{ B \Delta C \,:\, \Delta \in \R^{k,l} \},
\]
where $B\in\R^{n,k}$ and $C\in\R^{l,n}$ with $k,l<n$  are given matrices of full rank. Here, $\Pi^\cS Z = B B^\dagger Z C^\dagger C$,
where $B^\dagger$ and $C^\dagger$ are the Moore--Penrose pseudo-inverses of $B$ and $C$, respectively. 
\end{example}

\begin{example}[Toeplitz matrices] \rm If $\cS$ is the space of complex $n\times n$ Toeplitz matrices, then $\Pi^\cS Z$ is obtained by replacing in  each diagonal all the entries of $Z$ by their arithmetic mean. For real Toeplitz matrices, the same action is done on $\Re\, Z$. 
\end{example}

\begin{example} [Hamiltonian matrices] \rm If $\cS$ is the space of $2d\times 2d$ real Hamiltonian matrices, then $\Pi^\cS Z = J^{-1}\mathrm{Sym}(\Re(JZ))$, where $\mathrm{Sym}(\cdot)$ takes the symmetric part of a matrix and (here $I_d$ denotes the identity matrix)
\[
J=\begin{pmatrix}
0 & I_d \\ -I_d & 0
\end{pmatrix},
\]
for which $J^{-1}=J^\top=-J$. We recall that a real matrix $A$ is Hamiltonian if $JA$ is symmetric.
\end{example}

\subsection{Free gradient of the functional $F_\eps$}
To derive the gradient of the functional $F_\eps$, we need the derivative of the target eigenvalue $\lambda(A+\eps E(t))$ along paths of matrices $E(t)$, for $t$ in some interval. In the case of a simple eigenvalue, which is the situation we will consider in the following,
this derivative is obtained from the following well-known result.

\begin{lemma}[Derivative of simple eigenvalues (e.g. \cite{Kato})]\label{lem:eigderiv} %
Consider a continuously differentiable path of square complex matrices $M(t)$ for $t$ in an open interval $I$. Let $\lambda(t)$, $t\in I$, be a continuous path of simple eigenvalues of $M(t)$. Let $x(t)$ and $y(t)$ be left and right eigenvectors, respectively, of $M(t)$ to the eigenvalue $\lambda(t)$. Then, $x(t)^*y(t) \neq 0$ for $t\in I$ and $\lambda$ is continuously differentiable on $I$ with 
\begin{equation}
\dot{\lambda} = \frac{x^* \dot{M} y}{x^* y},\,
\end{equation}
where we omit dependence on time and indicate by dot differentiation wrt time.

Moreover, ``continuously differentiable'' can be replaced with ``analytic'' in the assumption and the conclusion.
\end{lemma}

Since we have  $x(t)^*y(t) \neq 0$, we can apply the  normalization
\begin{equation}\label{eq:scalxy}
\| x(t) \|=1, \ \ \|y(t) \| =1, \quad x(t)^* y(t) \text{ is real and positive.}
\end{equation}
The norm $\|\cdot\|$ is chosen as the Euclidean norm, and $x^*={\overline x}^\top$.
Clearly, a pair of left and right eigenvectors $x$ and
$y$ fulfilling \eqref{eq:scalxy} may be replaced by $\mu x$ and $\mu y$ for 
any complex $\mu$ of modulus $1$ without changing the property~\eqref{eq:scalxy}.

The following lemma will allow us to compute the steepest descent direction of the functional $F_\eps$
in $\C^{n,n}$, which means neglecting any structural constraint. For this reason we refer to it as the 
{\em free gradient} of the functional. 

\begin{lemma}[Free gradient] \label{lem:gradient}
Let $E(t)\in \C^{n,n}$, for $t$ near $t_0$, be a continuously differentiable path of matrices, with the derivative denoted by $\dot E(t)$.
Assume that $\lambda(t)$ is a simple eigenvalue of  $A+\eps E(t)$ depending continuously on $t$, with associated eigenvectors
$x(t)$ and $y(t)$ satisfying \eqref{eq:scalxy}, and let the eigenvalue condition number be
\begin{equation*}
\kappa(t) = \frac1{x(t)^* y(t)} > 0.
\end{equation*}
Then, $F_\eps(E(t))=f \bigl( \lambda(t) , \overline{\lambda(t)}  \bigr)$ 
is continuously differentiable w.r.t. $t$ and we have
\begin{equation} \label{eq:deriv}
\frac1{ \eps \kappa(t) } \,\frac{d}{dt} F_\eps(E(t)) = \Re \,\Bigl\langle  G_\eps(E(t)),  \dot E(t) \Bigr\rangle,
\end{equation}
where the (rescaled) gradient of $F_\eps$ is the rank-1 matrix
\begin{equation} \label{eq:freegrad}
G_\eps(E) = 2 f_{\clambda} \, x y^* \in \C^{n,n} \qquad \mbox{\rm with} \ f_{\clambda} = \frac{\partial f}{\partial \clambda}(\lambda, \clambda)
\end{equation}
for the target eigenvalue $\lambda=\lambda(A+\eps E)$ and the corresponding left and right eigenvectors $x$ and $y$ normalized 
according to \eqref{eq:scalxy}.
\end{lemma}
\begin{proof}
We first observe that \eqref{ass:f} implies 
\[
f_{\clambda} = \conj{f_{\lambda}} =\displaystyle\conj{\frac{\partial f}{\partial \lambda}(\lambda, \clambda)}.
\]
Using Lemma \ref{lem:eigderiv}, we obtain that $F_\eps(E(t))$ is differentiable with
\begin{eqnarray} \nonumber
\frac{ d }{dt} F_\eps \left( E(t) \right) & = & f_\lambda \,\dot \lambda + f_{\clambda} \,\overline{\dot \lambda} \\
& = & \frac{\eps}{x^*y}  \left( f_\lambda\, x^* \dot{E} y + f_{\clambda} \;\conj{x^* \dot{E} y} \right) = 
\frac{ \eps}{x^*y} \, 2\,\Re \left( f_\lambda\, x^* \dot{E} y \right),
\end{eqnarray}
where  we omit the omnipresent dependence on $t$ on the right-hand side.
Noting that
\[
\Re \bigl( f_\lambda \,x^*\dot{E} y  \bigr) = \Re\, \bigl\langle \conj{f_{\lambda}}\, x y^* , \dot{E}  \bigr\rangle,
\]
we obtain \eqref{eq:deriv}--\eqref{eq:freegrad}.
\end{proof}

\begin{example} \label{ex:G} \rm 
For 
\[
f(\lambda,\clambda) =-\tfrac12(\lambda+\clambda)=-\Re\,\lambda
\] 
we have $2 f_{\clambda}=-1$ and hence $G_\eps(E)=-xy^*$,
which is nonzero for all $\lambda$. For 
\[
f(\lambda,\clambda) = -\tfrac12 |\lambda|^2 = -\tfrac12 \lambda \clambda
\] 
we have 
$2 f_{\clambda} =  -\lambda$.
In this case 
$
G_\eps(E) =  - \lambda \,x y^*,
$
which is nonzero whenever $\lambda\ne 0$. 
\end{example}

\subsection{Projected gradient}

The optimization problem \eqref{eq:optimizS} is set on the manifold $\cS_1 = \{ E \in \cS, \ \| E \|_F = 1 \}$.

\subsubsection*{Preserving the structure}

Consider a smooth path of {\it structured} matrices $E(t)\in\cS$. Since then also $\dot E(t)\in\cS$, we have by Lemma~\ref{lem:gradient}
\begin{equation} \label{eq:deriv-S}
\frac1{ \eps \kappa(t) } \,\frac{d}{dt} F_\eps(E(t)) = \Re\bigl\langle  G_\eps^\cS(E(t)),  \dot E(t) \bigr\rangle
\end{equation}
with the rescaled structured gradient
\begin{equation}\label{gradient-S}
G_\eps^\cS(E) := \Pi^\cS \, G_\eps(E) = \Pi^\cS(2 f_{\clambda} \,xy^*) \in \cS,
\end{equation}
which is the projection onto $\cS$ of a rank-1 matrix.

\subsection*{Preserving the unit norm}

To comply with the constraint ${\| E(t) \|_F^2 =1}$, we must have 
\begin{equation} \label{eq:normconstr}
 0 = \frac12\,\frac d{dt}\| E(t) \|_F^2= \Re\, \Big\langle E(t), \dot E(t) \Big\rangle.
\end{equation}
In view of Lemma~\ref{lem:gradient} we are thus led to the following constrained optimization problem for the admissible direction of steepest descent.

\begin{lemma}[Direction of steepest admissible descent]
\label{lem:opt} 
Let $E \in \cS,\  G\in\C^{n,n}$ with ${\|E\|_F=1}$ and the orthogonal projection of $G$ onto $\cS$, $G^\cS = \Pi^\cS G \neq 0$. 
A solution of the optimization problem
\begin{eqnarray}
\nonumber
Z_\star  & = & \arg\min_{Z \in \C^{n,n}} \ \Re\,\langle  G,  Z \rangle 
\\
\nonumber
{\rm subj. to} &  & Z \in \cS 
\\
{\rm and}      &  & \Re\,\langle E, Z \rangle=0
\label{eq:opt}
\\
\nonumber
{\rm and}      &  & \|Z\|_F=1 \qquad \mbox{\rm (for uniqueness)}
\end{eqnarray}
is given by 
\begin{align}
\mu Z_\star & =  - G^\cS + \Re\hspace{1pt}\langle G^\cS, E \rangle\, E,
\label{eq:Eopt}
\end{align}
where  $\mu$ is 
the Frobenius norm of the matrix on the right-hand side. 
The solution is unique if $G^\cS$ is not a multiple of $E$.
\end{lemma}
\medskip 

\begin{proof}
The result follows on noting that the real part of the complex inner product on $\C^{n,n}$ is a real inner product on $\R^{2n,2n}$, and the real inner product with a given vector (which here is a matrix) is maximized over a subspace by orthogonally projecting the vector onto that subspace. The expression in (\ref{eq:Eopt}) is the orthogonal projection of $-G^\cS$ to the orthogonal complement of the span of $E$, which is the tangent space at $E$ of the manifold of matrices of unit Frobenius norm.
Since $E,G^\cS\in\cS$ in \eqref{eq:Eopt}, also $Z_\star$ is in $\cS$.
\end{proof}
\subsection{Constrained gradient flow}
Lemmas~\ref{lem:gradient} and~\ref{lem:opt} show that the admissible direction of steepest descent of the functional $F_\eps$ at a matrix 
$E \in \cS$ of unit Frobenius norm is given by the positive multiples of the matrix $-G_\eps^\cS(E) + \Re\,\langle G_\eps^\cS(E), E \rangle E$. 

This leads us to consider the (rescaled) {\it gradient flow on the manifold of $n\times n$ structured matrices in  $\cS$ of unit Frobenius norm}:
\begin{equation}\label{ode-E-S}
\dot E = - G_\eps^\cS(E) + \Re\,\langle G_\eps^\cS(E), E \rangle E.
\end{equation}

By construction of this ordinary differential equation, we have that $\dot E \in \cS$ and $\Re \langle E, \dot E \rangle =0$ along its solutions, and so both the structure $\cS$ and the Frobenius norm $1$ are conserved.
As we follow the admissible direction of steepest descent of the functional $F_\eps$ along solutions $E(t)$ of this differential equation, we obtain the monotonicity property stated in the next section.

\subsubsection*{Monotonicity} 
Assuming simple eigenvalues almost everywhere along the trajectory,
we have the following monotonicity property.
\begin{theorem}[Monotonicity] \label{thm:monotone}
Let $E(t)$ of unit Frobenius norm satisfy the differential equation {\rm (\ref{ode-E-S})}.
Then,
\begin{equation}
\frac{d}{dt} F_\eps (E(t))  \le  0.
\label{eq:pos-S}
\end{equation}
\end{theorem}
\begin{proof}
Although the result follows directly from Lemmas~\ref{lem:gradient} and \ref{lem:opt}, we compute the derivative directly. From Lemma~\ref{lem:gradient} and \eqref{ode-E-S} we have
\begin{align}
\nonumber
\frac1{\eps\kappa} \,\frac{d}{dt} F_\eps(E(t)) &= \Re \, \langle G_\eps^\cS(E), \dot E \rangle 
= \Re \, \langle G_\eps^\cS(E), -G_\eps^\cS(E) + \Re\,\langle G_\eps^\cS(E), E \rangle E \rangle
\\
\label{c-s}
&=  -\| G_\eps^\cS(E) \|_F^2 + \bigl( \Re\,\langle G_\eps^\cS(E), E \rangle \bigr)^2 \le 0.
\end{align}
The final inequality holds true by the Cauchy--Schwarz inequality and because $E$ has Frobenius norm $1$. 
\end{proof}

\subsubsection*{Stationary points}

A remarkable property of stationary points of \eqref{ode-E-S} is that they are projections onto $\cS$ 
of rank-$1$ matrices.

\begin{theorem} \label{thm:stat}
Let $E_\star\in\cS$ with $\| E_\star\|_F=1$ be such that the target eigenvalue $\lambda(A+\eps E_\star)$ is simple
and $G_\eps^\cS(E_\star)\ne 0$. Let $E(t)\in \cS$ be the solution of \eqref{ode-E-S} passing through $E_\star$. 
Then the following are equivalent: 
\begin{equation}\label{stat-S}
\begin{aligned}
1. \quad & 
\displaystyle\frac{ d }{dt} F_\eps \left( E(t) \right)  = 0
\\
2. \quad & 
\text{$E_\star$ \rm is a stationary point of the differential equation \eqref{ode-E-S}}.
\\
3. \quad & 
\text{$E_\star$ \rm is a real multiple of $G_\eps^\cS(E_\star)$.}
\end{aligned}
\end{equation}
\end{theorem}

\begin{proof} Clearly, 3.~implies 2., which implies 1. Since the Cauchy--Schwarz inequality in \eqref{c-s} is strict unless
$E$ is a real multiple of $G_\eps^\cS(E)$, we obtain that 1. implies 3. 
\end{proof}
\smallskip

As a consequence, optimizers of \eqref{eq:optimizS} are projections onto $\cS$ of rank-1 matrices. This motivates us to search for a differential equation on the manifold of rank-$1$ matrices that leads to the same stationary points.

\section{Rank-1 differential equation} \label{sec:rank1}
In this section we consider a differential equation on the manifold $\M_1$ of rank-1 matrices, which is shown to lead to the same stationary points as the structure- and norm-constrained gradient flow \eqref{ode-E-S}. This differential equation, reformulated for the factors of the rank-1 matrices, is to be numerically solved into a stationary point.

\subsection{A rank-1 projected differential equation}
Solutions of \eqref{ode-E-S} can be written as 
$
E(t)=\Pi^\cS Z(t)
$, 
where $Z(t)$ solves
\begin{equation}\label{ode-E-S-Z}
\dot Z = -G_\eps(E) + \Re \langle G_\eps(E), E \rangle Z.
\end{equation}
We note that $\Re\langle \Pi^\cS Z, \Pi^\cS \dot Z \rangle=0$ if $\|\Pi^\cS Z\|_F=1$, so that the unit Frobenius norm of $E(t)=\Pi^\cS Z(t)$ is conserved. 

As every stationary point of this differential equation is of rank 1, we project the right-hand side onto the tangent space $T_Y\cM_1$ 
at $Y$ belonging to the manifold of complex rank-1 matrices $\cM_1=\cM_1(\C^{n,n})$ and consider instead the projected differential equation with solutions of rank 1:
 \begin{equation}\label{ode-E-S-1}
\dot Y = -P_Y G_\eps(E) + \Re \langle P_Y G_\eps(E), E \rangle Y \quad\text{ with }\ E=\Pi^\cS Y \quad\text{ of unit norm}.
\end{equation}
Here, $P_Y:\C^{n,n}\to T_Y\M_1$ is the orthogonal projection onto the tangent space $T_Y\M_1$, which for a rank-1 matrix $Y=\sigma uv^*$ with $\|u\|=\|v\|=1$ is given as (see \cite{KL07})
\begin{equation} \label{PY}
P_Y(Z) = Z- (I-uu^*)Z(I-vv^*).
\end{equation}
It is useful to note that $P_Y(Y)=Y$.
For $E=\Pi^\cS Y$ of unit Frobenius norm in \eqref{ode-E-S-1}, we find
\[
\Re\langle E, \dot E \rangle =  \Re\langle  E, \dot Y \rangle 
=-\Re\langle E, P_Y G_\eps(E)\rangle + \Re \langle P_Y G_\eps(E), E \rangle \,\Re\langle E,Y\rangle =0,
\]
where we used that $\Re\langle E,Y\rangle= \Re\langle \Pi^\cS E,Y\rangle= \Re\langle E,\Pi^\cS Y\rangle=\Re\langle E,E\rangle =\| E \|_F^2 =1$. So we have
\[
\| E(t) \|_F =1 \qquad\text{for all }t.
\]

\subsection*{Stationary points}
The following theorem states that under a non-degeneracy condition, the differential equations  \eqref{ode-E-S} and~\eqref{ode-E-S-1} yield the same stationary points.
\begin{theorem}[Relating stationary points]
\label{thm:stat-S}  
\begin{itemize}
\item[\rm (a) ] Let $E\in\cS$ of unit Frobenius norm be a stationary point of the gradient system \eqref{ode-E-S} 
that satisfies $\Pi^\cS G_\eps(E)\ne 0$. Then, $E=\Pi^\cS Y$
for a certain matrix $Y\in\cM_1$ that is a stationary point of the differential equation \eqref{ode-E-S-1}.

\item[\rm (b) ] Conversely, let $Y\in\cM_1$  be a stationary point of the differential equation \eqref{ode-E-S-1} such that $E=\Pi^\cS Y$
has unit Frobenius norm and $P_Y G_\eps(E)\ne 0$. Then, $P_Y G_\eps(E)=G_\eps(E)$, $\,Y$ is a nonzero real multiple of $G_\eps(E)$, and $E$ is a stationary point of the gradient system \eqref{ode-E-S}.
\end{itemize}
\end{theorem}
\medskip

\begin{proof} Let $G=G_\eps(E)$ in this proof for short.

(a) By \eqref{stat-S}, $E=\mu^{-1} \Pi^\cS G$ for some nonzero real $\mu$. Then, $Y:=\mu^{-1}G$ is of rank 1 and we have $E=\Pi^\cS Y$. We further note that $P_Y G = \mu P_Y Y = \mu Y = G$. We thus have
\[
-P_Y  G + \Re\langle P_Y G, E \rangle Y  = -G + \Re\langle G,E \rangle Y.
\]
Here we find that
\[
 \Re\langle G,E \rangle= \Re\langle \Pi^\cS G,E \rangle =  \Re\langle \mu E,E \rangle = \mu \| E \|_F^2 = \mu.
\]
So we have
\[
-G + \Re\langle G,E \rangle Y = -G + \mu Y =0
\]
by the definition of $Y$. This shows that $Y$ is a stationary point of \eqref{ode-E-S-1}.

(b) We show that $Y$ is a nonzero real multiple of $G$. By Theorem~\ref{thm:stat}, 
$E$ is then a stationary point of the differential equation \eqref{ode-E-S}.

For a stationary point $Y$ of \eqref{ode-E-S-1}, 
we have that $P_Y(G)$ is a nonzero real multiple of $Y$. 
Hence, in view of $P_Y(Y)=Y$, we can write $G$ as
\[
G=\mu Y + W, \quad\text{ where $\mu\ne 0$ is real and $P_Y(W)=0$.}
\]
Writing the rank-1 matrix $Y=\rho uv^*$ with $\rho \ne 0$ and $\|u\|=\|v\|=1$, we then have by \eqref{PY} that
\[
 W=W-P_Y(W)= (I-uu^*)W(I-vv^*).
\]
 On the other hand, $G=2 \overline f_\lambda xy^*$  is also of rank 1. So we have
\[
2 \overline f_\lambda xy^* = \mu uv^* + (I-uu^*)W(I-vv^*).
\]
Multiplying from the right with $v$ yields that $x$ is a complex multiple of $u$, and multiplying from the left by 
$u^*$ yields that $y$ is a complex multiple of $v$. Hence, $G$ is a complex multiple of $Y$. Since we already know 
that $P_Y(G)$ is a nonzero real multiple of $P_Y(Y)=Y$, it follows that $G$ is the same real multiple of $Y$. 

Thus stationary points $Y\in\cM_1$ of the differential equation \eqref{ode-E-S-1} are characterized as real multiples of $G$. Hence,
$E=\Pi^\cS Y$ is a real multiple of $\Pi^\cS G$, and by Theorem~\ref{thm:stat}, $E=\Pi^\cS Y$ is a stationary point of \eqref{ode-E-S}.
\end{proof}

\subsection*{Possible loss of monotonicity}
Since the projections $\Pi^\cS$ and $P_Y$ do not commute, along solutions of \eqref{ode-E-S-1} we cannot guarantee the 
monotonicity property \eqref{eq:pos-S}  that we have for the constrained gradient system~\eqref{ode-E-S}.

However, in all our numerical experiments we  observed that  
starting with an initial datum given by the negative free gradient of the
considered functional \eqref{eq:freegrad}, i.e. $Y(0) = {}-G_\eps(0)$,
we always obtained a monotone convergence behavior to a (local) optimum. 
Only in very few cases, by starting from a randomly chosen initial datum, we were able 
to observe a nonmonotonic convergence. However the loss of monotonicity occurred only once,
after the first step, and monotonicity was recovered from the following step onwards. In the following section we will explain this behavior locally near a stationary point, but we have no theoretical explanation for the numerically observed monotonic behavior far from stationary points.


%
%
%

\subsection{Differential equations for the factors of rank-1 matrices}
\label{subsec:rank-1}

Equation \eqref{ode-E-S-1} is an abstract differential equation on the rank-1 manifold $\M_1$. 
We  write a rank-1 matrix $Y\in\cM_1$ in a non-unique way as
\[
Y=\rho uv^*,
\]
where $\rho\in \R,\ \rho>0$ and $u,v\in \C^n$ have unit norm. 

The following lemma shows how we can rewrite the rank-1 differential equation \eqref{ode-E-S-1}
in terms of differential equations for the  factors $u, v$ and an explicit formula for~$\rho$.

\begin{lemma}[Differential equations for the factors]
\label{lem:suv-1}
Every  solution $Y(t)\in \M_1$ of the rank-1 differential equation \eqref{ode-E-S-1} with $\| \Pi^\cS Y(t) \|_F=1$
can be written as $Y(t)=\rho(t)u(t)v(t)^*$ from the following differential equations for the factors $u$ and $v$ of unit norm,
\begin{align*}
\rho \dot u &=  -(I-uu^*) Gv - \displaystyle \tfrac\iu 2 \Im(u^*Gv)u,
\qquad
\\
\rho \dot v &=   -(I-vv^*) G^*u + \displaystyle \tfrac\iu 2 \Im(u^*Gv)v,
\end{align*}
where $G=G_\eps(E)$ for $E=\Pi^\cS Y = \rho\, \Pi^\cS (uv^*)$
and $\rho=1/\| \Pi^\cS(uv^*)\|_F$.
\end{lemma}

The positive factor $\rho$ on the left-hand sides of the differential equations for $u$ and $v$ only determines the 
speed with which  the trajectory is traversed, but has no influence on the trajectory itself. 

\begin{proof}
The equation for $\rho$ is obvious because
$1 = \| E \|_F = \rho \| \Pi^\cS(uv^*)\|_F$.
\\
We write the right-hand side of \eqref{ode-E-S-1} and use \eqref{PY} to obtain for $Y=\rho uv^*$
\begin{align*}
\dot Y&= -P_Y G + \Re\langle P_Y G,E \rangle Y \\
&= -\ (I-uu^*)Gvv^* - uu^* G(I-vv^*) - uu^*G vv^* + \Re\Big\langle P_Y G,E \Big\rangle Y\\
&= -\  \Bigl( (I-uu^*)Gvv^* + \tfrac\iu 2 \Im(u^*Gv)u \Bigr)v^*  
- u \Bigl( u^*G(I-vv^*) + \tfrac\iu 2 \Im(u^*Gv) v^* \Bigr) 
\\
&\quad\, - \Bigl( \Re (u^*Gv) +  \Re\langle P_Y G,E \rangle \rho \Bigr) uv^*.
\end{align*}
Since this is equal to $\dot Y = (\rho \dot u) v^* + u (\rho \dot v^*) + \dot \rho uv^*$, we can read off $\rho \dot u$, $\rho \dot v^*$ and $\dot\rho$ as the three terms in big brackets. This yields the stated differential equations for $u$ and~$v$ (and another one for $\rho$, which will not be needed). Note that $(d/dt)\| u\|^2 =2\,\Re(u^*\dot u)=0$ and analogously for $v$, so that the unit norm of $u$ and $v$ is conserved.
\end{proof}

We note that for $G=G_\eps(E)=2f_{\clambda}\,xy^*$ (see Lemma~\ref{lem:gradient}) and with $\alpha=u^*x$, $\beta=v^*y$
and $\gamma=2f_{\clambda}$, we obtain the differential equations 
\begin{equation}\label{ode-uv-short}
\begin{array}{rcl}
\rho \dot u &=&  \alpha\conj\beta\gamma\, u- \conj\beta\gamma \,x -\tfrac \iu2 \, \Im(\alpha\conj\beta\gamma)u
\\[3mm]
\rho \dot v &=&  \conj{\alpha}\beta\conj{\gamma}\,v -\conj{\alpha\gamma}\,y -\tfrac \iu2 \, \Im(\conj{\alpha}\beta\conj{\gamma})v.
\end{array}
\end{equation}

%
%
%

\subsection{Cases of interest} \label{sec:Suit}

The real dimension of the manifold of complex $n \times n$ rank-$1$ matrices of unit norm is $4 n -1$. 
Integrating \eqref{ode-E-S-1} instead of \eqref{ode-E-S} would be very appealing in those cases where $\dim \cS$ is significantly
larger than $4 n -1$. An important example is given by sparse matrices with a sparsity pattern with a number of nonzero elements of order $cn$ 
with $c > 4$ (and ideally much larger than $4$).

In the case of a real target eigenvalue the dimension of the manifold of real $n \times n$ rank-$1$ matrices of unit norm is $2 n -1$ so that 
for structured matrices it is meaningful to make use of \eqref{ode-E-S-1} only if $c > 2$.

Similarly, when considering matrices with prescribed range and co-range,
\[
\cS = \{ B \Delta C \,:\, \Delta \in \R^{k,l} \},
\]
where $B\in\R^{n,k}$ and $C\in\R^{l,n}$ with $k,l<n$, replacing the unknown matrix $\Delta$, which is a full $k \times l$ real matrix, 
by a rank-$1$ matrix, would significantly reduce the memory requirements when $k$ and $l$ are large. 
As for the computational cost, we may argue that the reduced number of variables may lead to faster convergence of the
numerical method.

On the contrary, for a real Toeplitz matrix, $\dim \cS = 2 n + 1$, which suggests to use \eqref{ode-E-S} in terms of the Toeplitz
coefficients, instead of the rank-1 differential equation \eqref{ode-E-S-1}. Instead, for a block Toeplitz matrix the use of \eqref{ode-E-S-1} appears preferable unless the blocks are very small.

\section{Local convergence to local minima of solutions to the rank-$1$ projected differential equation} \label{sec:locconv}

In this section we show that solutions of the rank-$1$ projected differential equation \eqref{ode-E-S-1} converge locally to strong local minima of the functional $F_\eps$. We first state the result, then formulate and prove a key lemma, and finally give the proof of the local convergence result.

\subsection{Statement of the local convergence result}

For the formulation of our local convergence result we need the following assumptions. Here, $\cS_1$ is the manifold of matrices in $\cS$ of unit Frobenius norm, and $\M_1$ is again the manifold of rank-1 matrices in $\C^{n,n}$. The first assumption is made on the structure space $\cS$. It excludes, in particular, spaces $\cS$ that are too low-dimensional: it requires 
$\text{dim}\,\cS \ge \text{dim}\, \M_1 = 4n-2$ (as before dim indicates the \emph{real} dimension of $\cS$).

\begin{assumption} 
\label{ass:S}
The restricted projection $\Pi^\cS\big|_{\M_1}: \cM_1 \to \Pi^\cS(\M_1) \subset \cS$ is a diffeomorphism, or equivalently:
\begin{itemize}
\item[(i)] If $E\in \Pi^\cS(\M_1)$, then there is a unique $Y\in\M_1$ such that $E=\Pi^\cS Y$.
\item[(ii)] The inverse map $E\mapsto Y$ is continuously differentiable.
\end{itemize}
\end{assumption}

\begin{remark} \rm
A comment on assumption \ref{ass:S} is helpful to understand it and justify it. Consider for example the structure $\cS$ of real matrices with a prescribed
sparsity pattern. Let $E \in \cS$ be given. Assume that for a given $\hat Y\in\M_1$ it holds that $E=\Pi^\cS \hat Y$. In principle in order 
to determine all solutions of the equation
\begin{equation}
\Pi^\cS Y = E
\end{equation}
we should form $Y = u v^*$ with $u,v \in \C^n$ and write a system of quadratic equations in the coefficients $\{ u_i \}_{i=1}^{n}$ and
$\{ v_j \}_{j=1}^{n}$ that reads
\begin{equation*}
\Re \left( u_i v_j^* \right) = E_{ij} \qquad \mbox{for all} \ i,j \in \cP
\end{equation*} 
where $\cP$ is the considered sparsity pattern, i.e. $\cP = \{ (i,j) : E_{ij} \not\equiv 0 \}$.

This gives a system of $p$ quadratic equations where $p = \#\cP$ is the number of entries of $E$ which are not prescribed to be zero.
In terms of the real variables $\Re(u_i), \Im(u_i), \Re(v_i), \Im(v_i)$ (indeed the first entry of $u$, if different from zero, can be 
chosen to be real and positive to guarantee uniqueness of the representation of $Y$ in terms of uniquely determined vectors $u$ and $v$, 
the previous is a system of $p$ quadratic equations in $4n-2$
variables so that we are allowed to generically expect the existence of the only solution $\hat Y$ if $p > 4n-2$.

We should also mention that the existence of a finite number of solutions would not affect our proof, but only the existence of a continuous
path of solutions in $\cM_1$ which would contradict the assumption that we make on the minimum, which we assume to be strong. 

About the usefulness of Assumption \ref{ass:S}, we note that if $p < 4n-2$ translating \eqref{ode-E-S} into a systems of ODE in terms of 
the $p$ nonzero entries of $E$ would be more convenient than integrating \eqref{ode-E-S}, which indicates that Assumption \ref{ass:S} is
very reasonable.
\end{remark} 
\medskip

The next assumption is made on the Hessian of the functional $F_\eps$ at a stationary point of the differential equation \eqref{ode-E-S}.

\begin{assumption} 
\label{ass:H}
Let $E_0\in \cS_1$ be a stationary point of the constrained gradient system \eqref{ode-E-S}. We assume that $E_0$ is a strong minimum of the functional $F_\eps$ on $\cS_1$, that is, the Hessian matrix $H_\eps(E_0)$ of $F_\eps$ at $E_0$ yields a positive definite quadratic form when restricted to the tangent space $T_{E_0}\cS_1$  of the manifold $\cS_1$ at $E_0$: there exists $\alpha>0$ such that
\begin{equation}\label{H-pd}
\langle Z, H_\eps(E_0) Z \rangle \ge \alpha \|Z\|^2 \qquad \forall\, Z\in T_{E_0}\cS_1.
\end{equation}
\end{assumption}

Under these assumptions we have the following result.

\begin{theorem}[Local convergence to a strong local minimum]
\label{thm:lc}
Under Assumption~{\rm \ref{ass:S}}, let  the rank-1 matrix $Y_0\in\M_1$
be a stationary point of the projected differential equation \eqref{ode-E-S-1} such that $E_0=\Pi^\cS Y_0\in \cS_1$ is of unit Frobenius norm and $P_{Y_0} G_\eps(E_0)\ne 0$.  We assume that $E_0$ satisfies Assumption~{\rm \ref{ass:H}}.

Then, for an initial datum $Y(0)$ sufficiently close to $Y_0$, the solution $Y(t)$ of \eqref{ode-E-S-1}
converges to $Y_0$ exponentially as $t \to \infty$. Moreover, $F_\eps \left( \Pi^\cS Y(t) \right)$ decreases monotonically with $t$ and converges exponentially to the local minimum value $F_\eps(E_0)$ as $t \to \infty$.
\end{theorem}

Note that $E_0=\Pi^\cS Y_0\in \cS_1$ is a stationary point of \eqref{ode-E-S} by Theorem~\ref{thm:stat-S} (b). So the assumption on $E_0$ reduces to the condition \eqref{H-pd} on the Hessian $H_\eps(E_0)$. 

The proof of Theorem~\ref{thm:lc} will be given in Section~\ref{subsec:proof-lc}.

\subsection{A basic lemma}

The following remarkable lemma provides the key to the proof of Theorem~\ref{thm:lc}.

\begin{lemma} \label{lem:loc}
\begin{samepage}
Let $Y_\star\in \M_1$ with $E_\star=\Pi^\cS Y_\star \in \cS$ of unit Frobenius norm. Let $Y_\star$ 
be a stationary point of the rank-1 projected differential equation \eqref{ode-E-S-1}, 
with an associated target eigenvalue $\lambda$ of $A+\eps E_\star$ that is simple.
Let $\delta$ be a sufficiently small positive number. Then, there exists $\bar \delta>0$ such that for all positive $\delta\le \bar\delta$ and all $Y \in \cM_1$ with $\| Y-Y_\star \| \le \delta$ and $E(Y)=\Pi^\cS Y$ of unit norm,
we have
\begin{equation}
\| P_Y G_\eps\left(E(Y) \right) - G_\eps\left(E(Y) \right) \| \le C \delta^2 
\end{equation}
with $C$ independent of $\delta$.
\end{samepage}
\end{lemma}
\begin{proof}
Let us consider a smooth path $Y(\tau) = u(\tau) v(\tau)^* \in \cM_1$ ($u(\tau), v(\tau) \in \C^n$) with 
\begin{align*}
& Y(0) = Y_\star = \alpha G_\star \quad \mbox{\rm for some real} \ \alpha \quad  \mbox{and} 
\\
& G_\star = G_\eps \left( E(Y(0)) \right) = 2 \overline f_\lambda xy^*,
\end{align*}
where $(\lambda,x,y)$ is the target eigentriplet
of $A + \eps \Pi^\cS Y(0)$ associated to the target eigenvalue~$\lambda$.

We indicate by $u,v,x,y$ and $\lambda$ (and later $\du,\dv,\dx$ and $\dy$) the associated functions of $\tau$ at $\tau=0$, i.e. in
correspondence of the stationary point.

We assume that $\| Y(\tau) - Y(0) \| \le \delta$ for $\tau \in [0,\delta]$ with $\delta$ such
that $\lambda(\tau)$ remains simple, and let 
\[
G(\tau) = G_\eps \left( E(Y(\tau)) \right) = 2 \overline f_\lambda(\tau)  x(\tau) y(\tau)^*.
\]

By the given assumptions all quantities are smooth w.r.t. $\tau$.
In particular, for a simple eigenvalue, under a smooth matrix perturbation, the derivatives of the associated eigenvectors
$\dx(\tau)$ and $\dy(\tau)$ - under the assumed normalization \eqref{eq:scalxy}  - are given by (see e.g. \cite{MS88,GL22})
\begin{equation*}\label{deigvec}
\begin{array}{rcl}
\displaystyle 
\frac{1}{\eps}\, \dx^*(\tau) & = & - x^*(\tau) \Pi_S \dot E(Y(\tau)) N(\tau)  + \Re\left(x(\tau)^* \Pi_S \dot E(Y(\tau)) N(\tau) x(\tau)\right) x(\tau)^*,
\\[3mm]
\displaystyle 
\frac{1}{\eps} \,\dy(\tau)  & = &  - N(\tau)   \Pi_S \dot E(Y(\tau)) y(\tau)  + \Re\left(y(\tau)^* N(\tau) \Pi_S \dot E(Y(\tau)) \right) y(\tau) ,
\end{array}
\end{equation*}
where  
$N(\tau)$ is the group inverse of  $ A + \eps \left( E(Y(\tau)) \right) - \lambda(\tau) I$. 
Note that by the simplicity of $\lambda$, $N(\tau)$ and thus also $\dx$ and $\dy$ as well as their derivatives are bounded.
%

With the formula \eqref{PY} for the projection $P_Y$, we thus have for $0 \le \tau \le \delta \ll 1$ the following first order expansion  (where the dot indicates here differentiation with respect to~$\tau$):
\begin{eqnarray}
&& P_{Y(\tau)} \Bigl( f_\clambda(\tau)  x(\tau) y(\tau)^* \Bigr) = \left( f_\clambda + \tau \dot{f}_\clambda \right) \, \cdot
\nonumber
\\[2mm]
&& 
\Biggl( \Bigl( x x^* + \tau \left(\du x^* + x \du^* \right) \Bigr)\, 
\Bigl( x y^* + \tau \left(\dx y^* + x \dy^* \right) \Bigr) +
\Biggr.
\nonumber
\\[2mm]
&& 
\Bigl( x y^* + \tau \left(\dx y^* + x \dy^* \right) \Bigr)\,
\Bigl( y y^* + \tau \left(\dv y^* + y \dv^* \right) \Bigr) -
\nonumber
\\[2mm]
&& 
\Biggl.
\Bigl( x x^* + \tau \left(\du x^* + x \du^* \right) \Bigr)\, 
\Bigl( x y^* + \tau \left(\dx y^* + x \dy^* \right) \Bigr)\,
\Bigl( y y^* + \tau \left(\dv y^* + y \dv^* \right) \Bigr) \Biggr) + \bigo(\tau^2) = 
\nonumber
\\
&& f_\clambda x y^* + \tau \left( \dot{f}_\clambda x y^* + f_\clambda x \dy^*  + f_\clambda \dx y^* \right) + \bigo(\tau^2).
\label{eq:firstorder}
\end{eqnarray}
Consequently, \eqref{eq:firstorder} has the same first order expansion as  
\begin{equation*}
f_\clambda(\tau)  x(\tau) y(\tau)^* = f_\clambda  x y^* + 
\tau \left( \dot{f}_\clambda x y^* + f_\clambda x \dy^*  + f_\lambda \dx y^* \right) + \bigo(\tau^2),
\end{equation*}
which yields the result. 
\end{proof}

\subsection{Proof of Theorem~\ref{thm:lc}}
\label{subsec:proof-lc}
With $E(t)=\Pi^\cS Y(t)\in \cS_1$, the differential equation \eqref{ode-E-S-1}  for $Y(t)\in \cM_1$ is equivalent to
$$
\dot E = - \Pi^\cS P_Y G_\eps(E) + \Re \langle  \Pi^\cS P_Y G_\eps(E), E \rangle E.
$$
By Lemma~\ref{lem:loc}, this can be rewritten as a perturbation to the constrained gradient system~\ref{ode-E-S} (recall that $G_\eps^\cS=\Pi^\cS G_\eps$):
$$
\dot E = - G_\eps^\cS(E) + \Re \langle G_\eps^\cS(E), E \rangle E + D
\qquad
\text{with}
\quad\
\| D(t) \| = \bigo( \| Y(t)-Y_\star \|^2).
$$
By Assumption~\ref{ass:S} (ii), this bound further implies
$$
\| D(t) \| = \bigo( \| E(t)-E_\star \|^2).
$$
The orthogonal projection $\widehat \Pi_E$ of $Z \in \C^{n,n}$ onto the tangent space $T_E\cS_1$ at $E\in \cS_1$ is given by
$$
\widehat \Pi_E Z = \Pi^\cS Z - \Re \langle \Pi^\cS Z, E \rangle E.
$$
We write 
\[
\widehat G(E) = \widehat \Pi_E G_\eps(E) = G_\eps^\cS(E) - \Re \langle G_\eps^\cS(E), E \rangle E 
\]
for short. We have
$$
\frac 12\, \frac{d}{d t}\, \| E(t)-E_\star \|^2 = \Re \langle E - E_\star, \dot E \rangle = \Re \langle E - E_\star, - \widehat G(E)+D \rangle.
$$
Since $\widehat G(E_\star)=0$ and 
\[
E-E_\star= \widehat \Pi_{E_\star}(E-E_\star) + \bigo(\| E- E_\star\|^2),
\]
which is due to the fact that both $E$ and $E_\star$ lie on $\cS_1$  so that letting $\delta := \| E - E_\star \|$ the orthogonal projection
$\widehat \Pi_{E_\star}(E-E_\star)$ onto the tangent plane at $E_*$ is $\delta^2$-close to $E-E_*$, we find
$$
\widehat G(E) = \widehat G(E) - \widehat G(E_\star) = \widehat \Pi_{E_\star} H_\eps(E_\star)\widehat \Pi_{E_\star} (E-E_\star) + \bigo(\| E-E_\star \|^2),
$$
where $H_\eps(E_\star)$ is the Hessian matrix of the functional $F_\eps$ at $E_\star$. By Assumption~\ref{ass:H}, $H_\eps(E_\star)$ is positive definite on $T_{E_\star}\cS_1$. So we obtain
\begin{eqnarray*}
& & \Re \langle E - E_\star, - \widehat G(E)+D \rangle 
\\
&=& \Re \langle \widehat \Pi_{E_\star}(E-E_\star) + \bigo(\| E- E_\star\|^2), -\widehat \Pi_{E_\star} H(E_\star)\widehat \Pi_{E_\star}(E-E_\star)+ \bigo(\| E- E_\star\|^2) \rangle
\\
&=&  - \langle \widehat \Pi_{E_\star}(E-E_\star), -H(E_\star)\widehat \Pi_{E_\star}(E-E_\star) \rangle + \bigo(\| E- E_\star\|^3)
\\
&\le & - \alpha \| \widehat \Pi_{E_\star}(E-E_\star) \|^2 + \bigo(\| E- E_\star\|^3)
\\
&\le & - \tfrac12\alpha \| E-E_\star \|^2,
\end{eqnarray*}
provided that $E$ is sufficiently close to $E_\star$. This yields that $\| E(t)-E_\star\|$ decreases monotonically with growing $t$ and converges exponentially fast to $0$ as $t\to\infty$.

Similarly we obtain, with the projected Hessian $\widehat H (E_\star) = \widehat \Pi_{E_\star} H_\eps(E_\star)\widehat \Pi_{E_\star}$ for short,
\begin{eqnarray*}
& & \frac 1{\kappa\eps} \, \frac{d}{dt} F_\eps(E(t)) = \Re \langle G_\eps(E),\dot E \rangle = \Re \langle \widehat G(E),\dot E \rangle
= \Re \langle \widehat G(E),- \widehat G(E)+D \rangle
\\
& = &- \| \widehat H(E_\star) \widehat \Pi_{E_\star}(E-E_\star) \|^2 +  \bigo(\| E- E_\star\|^3)
\\
& \le & - \alpha^2 \| \widehat \Pi_{E_\star}(E-E_\star) \|^2 +  \bigo(\| E- E_\star\|^3)
\\
& \le & - \tfrac12\alpha^2 \| E-E_\star \|^2,
\end{eqnarray*}
provided that $E$ is sufficiently close to $E_\star$. We conclude that $F_\eps(E(t))$ decreases monotonically and exponentially to $F_\eps(E_\star)$.
\qed

\section{Numerical integration by a splitting method}
\label{sec:proto-num}
%
We need to integrate numerically the differential equations (\ref{ode-uv-short}).
The objective here is not to follow a particular trajectory accurately, but to arrive quickly at a stationary point.
The simplest method is the normalized Euler method, where the result after an Euler step (i.e., a steepest descent step) is normalized to unit norm for both the $u$- and $v$-component.
This can be combined with a strategy to determine the step size adaptively. 
We found, however, that a more efficient method is obtained with a  splitting method instead of the Euler method.

\subsection{Splitting}
The splitting method consists of a first step applied to the differential equations
\begin{equation}\label{ode-uv-horiz}
\begin{array}{rcl}
 \rho\dot u &=&  \alpha\conj\beta\gamma\, u- \conj\beta\gamma \,x  
\\[3mm]
 \rho\dot v &=&  \conj{\alpha}\beta\conj{\gamma}\,v -\conj{\alpha\gamma}\,y  
\end{array}
\end{equation}
followed by a step for the differential equations
\begin{equation}\label{ode-uv-rot} 
\begin{array}{rcl}
 \rho\dot u &=& -\displaystyle \tfrac \iu2 \, \Im(\alpha\conj\beta\gamma)u 
\\[2mm]
 \rho\dot v &=& + \displaystyle \tfrac \iu2 \, \Im(\alpha\conj\beta\gamma) v.
\end{array}
\end{equation}
Note that the second differential equation is a mere rotation of $u$ and $v$.

As is very unusual, this splitting method preserves stationary points. 

\begin{lemma}[Stationary points] \label{lem:stat-split}
If $(u,v)$ is a stationary point of the differential equations \eqref{ode-uv-short}, then it is also a stationary point of the differential equations
\eqref{ode-uv-horiz} and \eqref{ode-uv-rot}.
\end{lemma}

\begin{proof} If $(u,v)$ is a stationary point of \eqref{ode-uv-short}, then $u$ is proportional to $x$ and $v$ is proportional to $y$. Hence,
$x=\alpha u$ and $y=\beta v$. This implies that $(u,v)$ is a stationary point of \eqref{ode-uv-horiz}, and hence also of  \eqref{ode-uv-rot}.
\end{proof}

\subsection{Fully discrete splitting algorithm}
Starting from vectors $u_k,v_k$ of unit norm and
\begin{equation}\label{rho-k}
\rho_k = \frac1{ \| \Pi^\cS(u_kv_k^*)\|_F}, 
\end{equation}
we denote by $x_k$ and $y_k$ the left and right eigenvectors to the target 
eigenvalue $\lambda_{k}$ of $A+\eps \rho_k \Pi^\cS(u_k v_k^*)$, and set 
\begin{equation}\label{alpha-beta-gamma-n}
\alpha_k=u_k^*x_k, \qquad \beta_k=v_k^*y_k, \qquad \gamma_k = 2f_{\clambda_k}. 
\end{equation}
We apply the Euler method with step size $h$  to \eqref{ode-uv-horiz} to obtain
\begin{equation}
\label{eul-horiz}
\begin{array}{rcl}
{\widehat u}(h) &=& u_k + (h/\rho_k)\left( \alpha_k\conj\beta_k\gamma_k\, u_k - \conj\beta_k\gamma_k \,x_k \right)
\\[2mm]
{\widehat v}(h) &=& v_k + (h/\rho_k)\left( \conj\alpha_k\beta_k\conj\gamma_k\,v_k -\conj{\alpha_k\gamma_k}\,y_k \right),
\end{array}
\end{equation}
followed by a normalization to unit norm 
\begin{equation} \label{eq:normal}
\widetilde u(h)=\frac{\widehat u(h)}{\|\widehat u(h)\|},\quad
\widetilde v(h)=\frac{\widehat v(h)}{\|\widehat v(h)\|}.
\end{equation}


Then, as a second step, we integrate the  rotating differential equations \eqref{ode-uv-rot} by setting, 
with $ \vartheta = -\displaystyle \tfrac 1{2\rho_k}\, \Im\! ( \alpha_k\conj{\beta_k}\gamma_k )$,
\begin{equation} \label{eq:rotate}
u(h)=\e^{\iu \vartheta h} \, \widetilde u(h), \qquad
v(h)=\e^{{}-\iu \vartheta h} \, \widetilde v (h),
\end{equation} 
set
$\rho(h) = 1/{\| \Pi^\cS( u(h) v(h)^*) \|_F}$,
and compute the target eigenvalue $\lambda(h)$ of the perturbed matrix $A+\eps \rho(h)\Pi^\cS\bigl(u(h)v(h)^*\bigr)$.
We note that this fully discrete algorithm still preserves stationary points. 
%
%

One motivation for choosing this method is that near a stationary point, the motion is almost 
rotational since $x \approx \alpha u$ and $y \approx \beta v$. 
The dominant term determining the motion is then the rotational term on the right-hand side of 
(\ref{ode-uv-short}), which is integrated by a rotation in the above scheme (the integration would be exact if $\alpha,\beta,\gamma$ were constant). 

This algorithm requires in each step one computation of target eigenvalues and associated eigenvectors
of structure-projected rank-$1$ perturbations to the matrix $A$, which can be computed at moderate computational cost for 
large sparse matrices $A$ by using an implicitly restarted Arnoldi method 
(as implemented in ARPACK 
and used in the MATLAB function {\it eigs} \cite{LeSY98}).
%
%

\medskip
\begin{algorithm}[H] \label{alg_prEul}
\DontPrintSemicolon
\KwData{$A, \eps, \theta > 1, u_k \approx u(t_k), v_k \approx v(t_k)$, $h_{k}$ (proposed step size),\\
$\lambda_k$ (target eigenvalue of $A+ \Delta_k$ with $\Delta_k =\eps \Pi^\cS(u_k v_k^*)/\|\Pi^\cS(u_k v_k^*)\|_F)$}
\KwResult{$u_{k+1}, v_{k+1}$, $h_{k+1}$, $\lambda_{k+1}$}
\Begin{
\nl Initialize the step size by the proposed step size, $h=h_{k}$\; 
\nl Compute left/right eigenvectors 
$x_k, y_k$ of $A + \Delta_k$ to $\lambda_k$ such that $\| x_k \| = \| y_k \| = 1, x_k^* y_k > 0$\; 
\nl Compute $\alpha_k,\beta_k,\gamma_k$ by \eqref{alpha-beta-gamma-n} and $g_k$ by \eqref{g-formula}\;
\nl Initialize $f(h) = f_k$\;
\While{$f(h) \ge\max( f_k, f_k-h\theta g_k)$}{
\nl Compute $u(h), v(h)$ according to \eqref{eul-horiz}-\eqref{eq:rotate}\;
\nl Compute $\Delta(h)= \eps\rho(h) \Pi^\cS(u(h) v(h)^*)$ with $\rho(h)=1/\|\Pi^\cS(u(h) v(h)^*)\|_F$\;
\nl Compute $\lambda(h)$ target eigenvalue of $A + \Delta(h)$\; 
\nl Compute the value $f(h) = f\bigl( \lambda(h), \conj{\lambda(h)} \bigr)$\;
\If{$f(h) \ge \max(f_k,f_k-h\theta g_k)$}{Reduce the step size, $h:=h/\theta$}
}
\uIf{$\bigl(g_k\ge 0\ \text{\rm and}\ f(h) \ge f_k- (h/\theta) g_k\bigr) \ \text{\rm or } \bigl(g_k< 0\ \text{\rm and}\ f(h) \ge f_k- h\theta g_k\bigr)$}{Reduce the step size for the next step, $h_{\rm next}:=h/\theta$}
\uElseIf{$h=h_k$}{Set $h_{\rm next} := \theta h_k$ (augment the stepsize if no rejection has occurred)}
\Else{Set $h_{\rm next} := h_k$}
\nl Set $h_{k+1} := h_{\rm next}$, $\lambda_{k+1} := \lambda(h)$, and the starting values for the next step as 
$u_{k+1} := u(h)$, $v_{k+1} := v(h)$\;
\Return
}
\caption{Integration step for the rank-1 differential equation}
\end{algorithm}

\subsection{Step-size selection} 


We use an  Armijo-type line search strategy, adapted to the possibility that the functional $f(\lambda,\clambda)$ is not everywhere  reduced along the flow of the differential equation \eqref{ode-E-S-1} (even though this was never observed in our numerical experiments when we chose
the initial value $Y(0)$ as a positive multiple of the negative free gradient $-2 f_{\clambda} \, x y^*$ where $(\lambda,x,y)$ is the target eigentriplet of the matrix $A$).
By Lemma~\ref{lem:gradient}, the change of the functional along solutions of \eqref{ode-E-S-1} equals (with $G=G_\eps(E))$ and omitting the argument $t$ on the right-hand side)
\begin{align}
\frac{d}{dt} F_\eps(E(t)) &=\eps\kappa \, \Re\langle G_\eps(E),\dot E \rangle \nonumber
\\
&= -\eps\kappa \Bigl(  \Re\langle \Pi^\cS G, \Pi^\cS P_Y G \rangle - \Re\langle \Pi^\cS P_Y G, E\rangle\, \Re\langle \Pi^\cS G, E \rangle\Bigl) = : -g
\label{g-formula}
\end{align}
We write $g_k=g$ for the choice $E=E_k=u_k v_k^*$, $G=G_\eps(E_k)= 2f_{\clambda}(\lambda_k,\conj{\lambda_k})x_k y_k^*$, and $\kappa=\kappa_k=1/(x_k^*y_k)$.
Let 
$$
f_k = f(\lambda_k,\conj{\lambda_k}), \qquad f(h) = f(\lambda(h),\conj{\lambda(h)}).
$$
We accept the result of the step with step size $h$ if, for a given parameter $\theta >1$,
$$
f(h) < \max( f_k, f_k - h\theta g_k).
$$
If 
$
g_k\ge 0\text{ and } f(h) \ge f_k - (h/\theta) g_k,
$
or if
$
g_k < 0 \text{ and } f(h) \ge f_k - h\theta g_k,
$
then we reduce the step size for the next step to $h/\theta$. If the step size has not been reduced in the previous step, we try for a larger step size.
Algorithm \ref{alg_prEul} describes
the step from $t_{k}$ to $t_{k+1} = t_{k}+h_{k}$. 

\section{Application to structured matrix nearness problems} 
\label{sec:stabrad}

We consider  matrix nearness problems that are closely related to the eigenvalue optimization problems considered in this article. 
We pose the problem in the structure space $\cS$.
Let again $ A \in \C^{n,n}$ be a given matrix  and let $\lambda( A)\in\C$ be a target eigenvalue of~$ A$. We again consider the  smooth function
$f(\lambda,\clambda)$ satisfying \eqref{ass:f} that is to be minimized.
For a prescribed real number $r$ in the range of $f$ we assume that
$$
f(\lambda( A),\clambda( A)) > r,
$$
so that for sufficiently small $\eps>0$ we have $\ophi(\eps)>r$, where
\[
\ophi(\eps) :=\min\limits_{\Delta \in \cS,\, \| \Delta \|_F = \eps} f \left( \lambda\left(  A + \Delta \right), \clambda \left(  A + \Delta \right)  \right).
\]
The objective now is to find the smallest $\eps>0$ such that  $\ophi(\eps)=r$: 
\begin{equation}
\oeps = \min\bigl\{\eps > 0 \,:\, 
\ophi(\eps) \le r \bigr\}.
\label{eq:mnpb}
\end{equation}
Determining $\oeps$ is a one-dimensional root-finding problem for the function $\ophi$ that is defined by the considered 
eigenvalue optimization problem.

\subsection{Structured distances to singularity and  to instability} 
Let us consider two examples, with a peculiar difference. In the first case the problem reduces to the
search of the simple (unique) zero of a smooth function, while in the second case the function is not
smooth at its smallest zero, and (generically) vanishes identically right to it.  
\begin{example}[Structured distance to instability]
\label{ex:1}
Let $A$ be a Hurwitz matrix, i.e. with negative spectral abscissa
$\alpha( A)= \max\{\Re\, \lambda\,:\, \lambda \text{ is an eigenvalue of $A$} \} < 0$.
With the function 
$
f ( \lambda, \clambda) = -\tfrac12(\lambda + \clambda) = -\Re\, \lambda
$
and the target eigenvalue $\lambda$ given by the eigenvalue of largest real part, and $r=0$,
we arrive at the problem of computing the structured distance to instability of $A$ :
\begin{equation*}
\oeps = \min\{\eps > 0 \,:\, \alpha_\eps^\cS( A) = 0 \},
\quad \mbox{\rm where} \quad
\alpha_\eps^\cS( A) = \max\limits_{E \in \cS, \| E \|_F = 1} \alpha( A + \eps E)
\end{equation*}
is the $\eps$-pseudospectral abscissa with respect to the structure space $\cS$.
\end{example}

%

\begin{example}[Structured distance to singularity] 
\label{ex:2}
Let $A$ be a nonsingular matrix.
With $f( \lambda, \clambda) = \lambda\clambda = |\lambda|^2$  
and the target eigenvalue $\lambda$ given by the eigenvalue of smallest modulus, we arrive at 
the problem of computing the distance to singularity of $A$ :
\begin{equation*}
\oeps = \min \{\eps > 0 \,:\, \varrho_\eps^\cS( A) = 0 \},
\quad \mbox{\rm where} \quad
\varrho_\eps^\cS( A)= \min\limits_{E \in \cS, \| E \|_F = 1} \varrho( A + \eps E)
\end{equation*} 
where $\varrho(M)$ is the smallest modulus of eigenvalues of a matrix $M$.\footnote{Instead 
of eigenvalues of smallest modulus, we could take the smallest singular value.}
\end{example}

\subsection{Two-level iterative method}
\label{sec:two-level}

As in previous work (see e.g. \cite{GKL15,G16}), we use a two-level method:
\begin{itemize}

\item[(i) ] {\bf Inner iteration:\/} Given $\eps>0$, we aim to compute a  matrix $E(\eps) \in\cS$  
of unit Frobenius norm,  such that 
$F_\eps(E) = f \left( \lambda\left(  A + \eps E \right), \clambda \left(  A + \eps E \right)  \right)$ 
is minimized: 
\begin{equation} \label{E-eps-2l}
E(\eps) = \arg\min\limits_{E \in \cS, \| E \|_F = 1} F_\eps(E).
\end{equation}


\item[(ii) ] {\bf Outer iteration:\/} We compute the smallest positive value $\oeps$ with
\begin{equation} \label{eq:zero}
\ophi(\oeps)= r,
\end{equation}
where $\ophi(\eps)= F_\eps\left(E(\eps) \right) = f \left( \lambda\left(  A + \eps E(\eps) \right), \clambda \left(  A + \eps E(\eps) \right)  \right)$.
\end{itemize}

\subsection{Inner iteration}

The eigenvalue optimization problem \eqref{E-eps-2l} is precisely of the type studied in the previous sections.
To compute $E(\eps)$ for a given $\eps>0$, we integrate numerically either the ODE system \eqref{ode-E-S}
or \eqref{ode-E-S-1}; see Section \ref{sec:proto-num}.

The computational cost can be significantly reduced if we are able to compute efficiently $\Pi^\cS (Y)$
and the matrix vector multiplication $\Pi^\cS (Y) v$ (with $v \in \C^n$) which is typically used by an 
iterative eigensolver applied to $A + \eps \Pi^\cS (Y)$. 
This is true for example when $\cS$ is the set of matrices with a prescribed sparsity pattern.

\subsection{Outer iteration}

The outer iteration determines
the smallest positive solution of the one-dimensional root-finding 
problem \eqref{eq:zero}.
We make use of a locally quadratically convergent Newton-type method, 
which can be justified under appropriate regularity assumptions.
 It turns out that the derivative of $\ophi$ is then simply
\begin{equation} \label{eq:dphi}
\ophi'(\eps)= - \| \Pi^\cS G_\eps(E(\eps)) \|_F / (x(\eps)^*y(\eps)), 
\end{equation}
where $x(\eps)$ and $y(\eps)$ with $x(\eps)^*y(\eps)>0$ are the eigenvectors to the (simple) target eigenvalue $\lambda(\eps)$ of $A+\eps E(\eps)$; cf.~\cite{G16,GL17} for related derivative formulas. 
If the assumptions justifying this formula are not met, we can always resort to bijection. The algorithm we use is indeed a combined 
Newton / bisection approach, similar to \cite{GKL15,G16,GL17}.

\section{Illustrative examples} \label{sec:ill}

In this section we show the behavior of Algorithm \ref{alg_prEul}, which is based on the rank-1 differential equation \eqref{ode-E-S-1}, on a few interesting examples: two sparse matrices and an example with prescribed range and corange.

%
%



We start by considering two well-known sparse matrices.

\subsection{The matrix ORANI678 from the Harwell Boeing collection}

\begin{figure}[ht] \label{fig:Orani1}
\vspace{-4.6cm}
\centerline{
\includegraphics[scale=0.45]{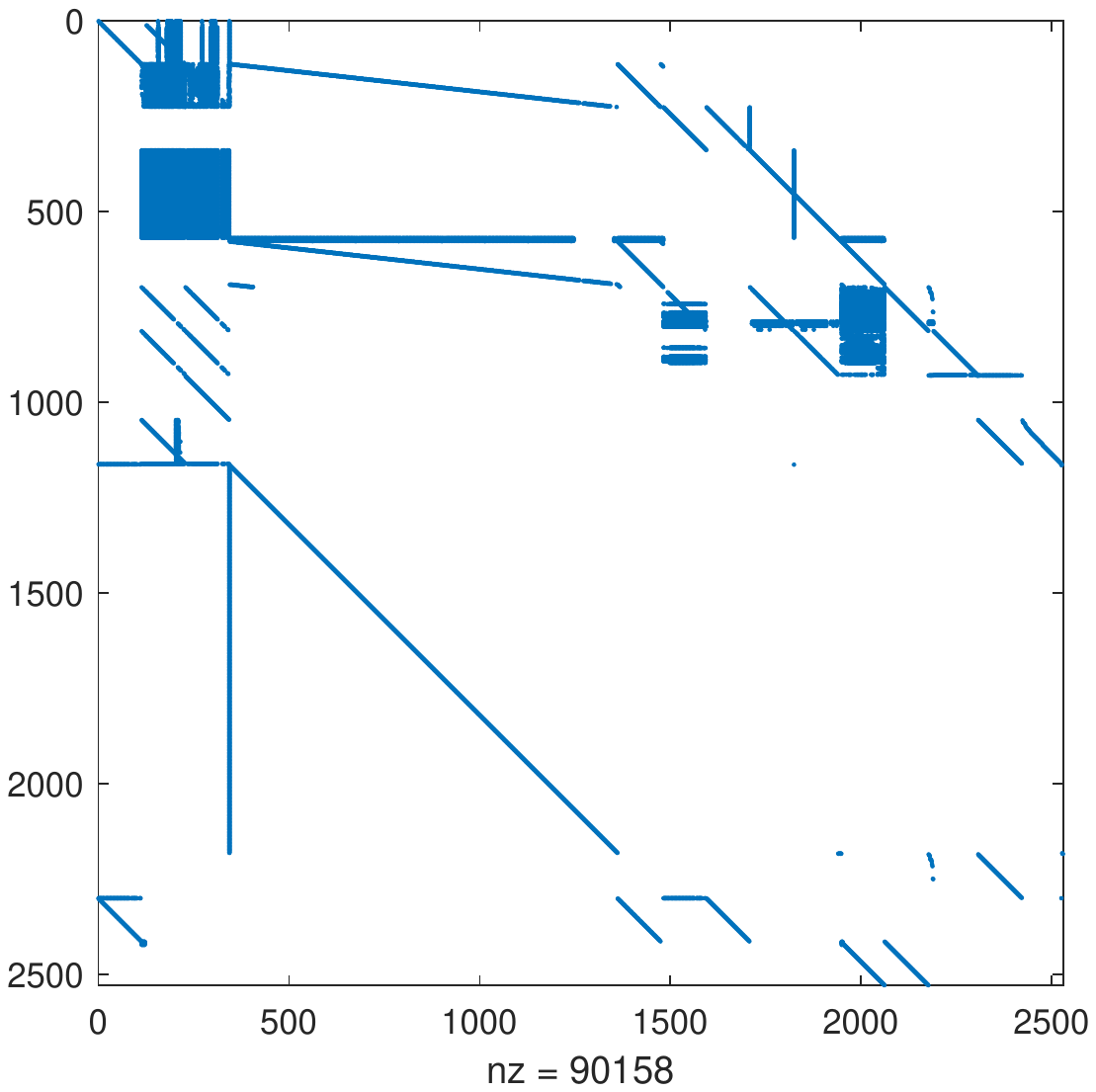}
\hspace{-28mm} 
\includegraphics[scale=0.45]{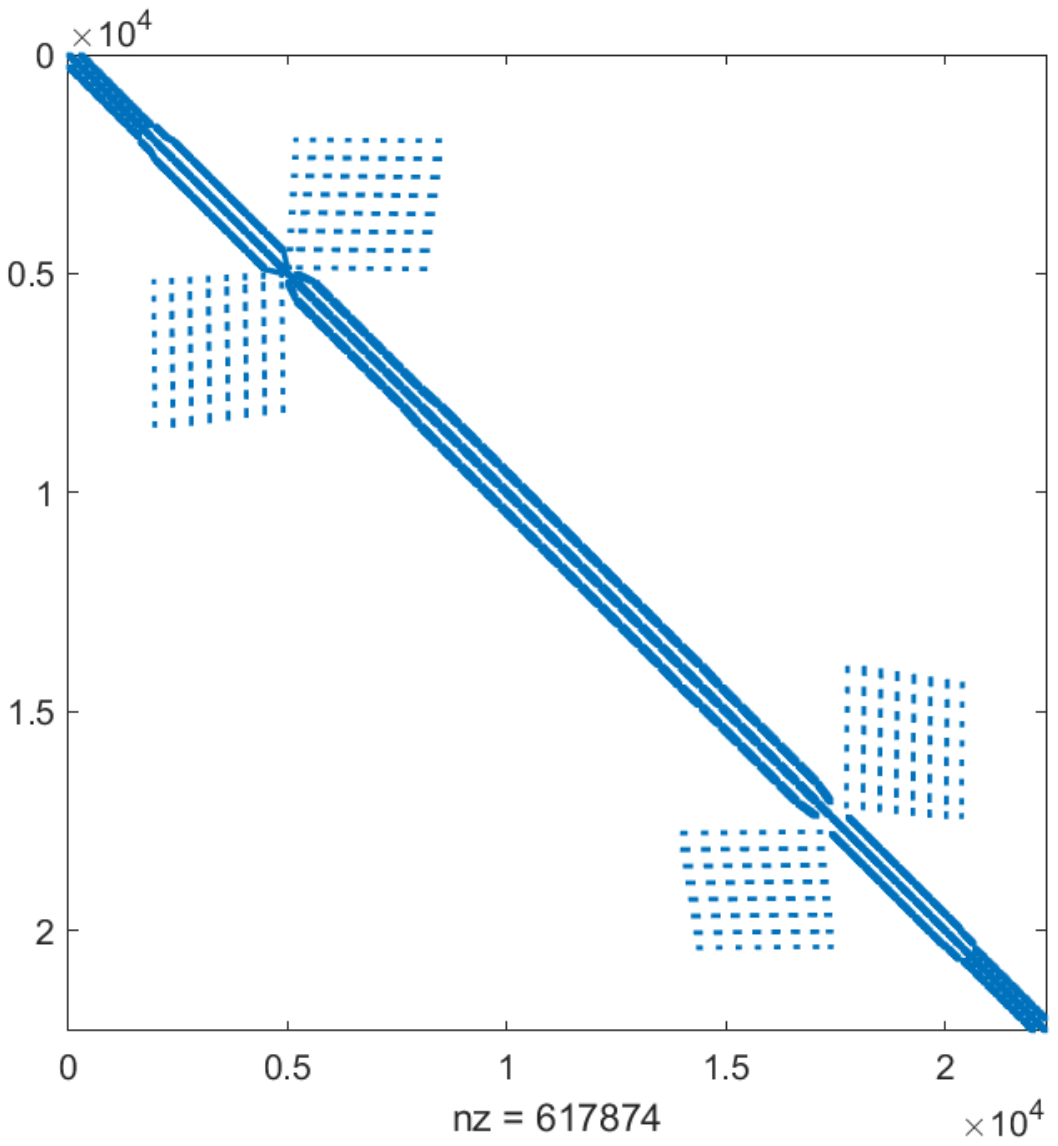}
}
\vspace{-3.6cm}
\caption{Sparsity patterns of the matrices ORANI678 (left) and FIDAPM11 (right).}
\end{figure}

The matrix $A$ is a sparse real unsymmetric square matrix taken from the set ECONAUS. It has dimension $n = 2529$ and a number
of nonzero entries $nz = 90158 \approx 40 n$.  Its sparsity pattern is plotted in Figure \ref{fig:Orani1}.

(i) 
We have set $\eps=1$ and applied our algorithms to the minimization problem \eqref{eq:optimizS} with
$
f(\lambda,\clambda) = -\frac12({\lambda+\clambda}) = - \Re(\lambda) 
$ 
and $\cS$ the space of real matrices with the sparsity pattern of $A$.
The target eigenvalue is the one with largest real part. We thus aim to compute
the \emph{structured $\eps$-pseudospectral abscissa} of $A$.


\begin{table}[th]
\begin{center}
\begin{tabular}{|l|l|}\hline
  $k$ &  $\Re(\lambda_k)$ \\
 \hline
\rule{0pt}{9pt}
\!\!\!\! 
$0$           &  $-1.232670912085709$ \\
$1$          &  $-1.745212357950066$ \\
$2$        &  $-1.917229680782718$ \\
$3$       &  $-2.076407232182272$ \\ 
$4$      &  $-{\bf 2}.249359154133923$ \\
$5$     &  $-{\bf 2.3}43018078428841$ \\
$6$     &  $-{\bf 2.3}43036033336665$ \\
$7$     &  $-{\bf 2.3}49611649664635$ \\
$8$     &  $-{\bf 2.350}556073486847$ \\
$9$     &  $-{\bf 2.3506}20017092603$ \\
$\ldots$ &   $\ldots$             \\
$25$ &   $-2.350634775262785$ \\  
 \hline
\end{tabular}
\vspace{2mm}
\caption{Computed values using Algorithm {\rm \ref{alg_prEul}} 
for the ORANI678 matrix.
\label{tab:ex_ORANI}}
%
\end{center}
\end{table}

We integrated \eqref{ode-E-S-1} by Algorithm \ref{alg_prEul} and obtained the results in Table \ref{tab:ex_ORANI}.
The main cost is the number of eigentriplets evaluations by the Matlab routine eigs \cite{LeSY98} and is given by $n_{\rm eig} = 38$.
The CPU time is around $2$ seconds.

For comparison we also integrated the full-rank ODE \eqref{ode-E-S} by the Euler method (gradient descent) with variable stepsize
and obtained a similar behavior. The number of eigentriplets evaluations is $n_{\rm eig} = 35$ and the final approximation
to the $\eps$-pseudospectral abscissa is $2.350634775261177$, which coincides with the value computed by the rank-$1$ method up to
the $11$-th digit. The CPU time is $1.6$ seconds.

Since $u$ and $v$ turn out to be real, the gain in terms of memory requirements for the rank-1 algorithm is $90158/5058 \approx 17.82$, which is a significant
reduction in the storage of the iterates.

(ii)
Setting next
$
f(\lambda,\clambda) = \lambda\clambda = |\lambda|^2 
$ 
and the target eigenvalue the one - say $\lambda_{\min}$ - with smallest modulus, we approximated the \emph{structured 
distance to singularity} of $A$. Given the convergence to a local optimizer of Algorithm \ref{alg_prEul} we obtain
this way an upper bound to this distance. 
An immediate lower bound is the unstructured distance $\sigma_{\min}(A)$, i.e. the smallest singular value, which is
equal to $0.0033388$. As we see in Table \ref{tab:ORdsing}, the effective structured distance to singularity is one
order of magnitude larger.


Applying a Newton-bisection method we obtained the results
shown in Table \ref{tab:ORdsing}. Since the function $\phi$ 
and its derivative (see \eqref{eq:dphi}) are computed
inexactly  (by Algorithm \ref{alg_prEul}), we do not observe quadratic convergence.
\begin{table}[hbt]
\begin{center}
\begin{tabular}{|l|l|l|l|}\hline
  $k$ & $\eps_k$ & $\phi(\eps_k)$ & $\#$ eigs \\
 \hline
\rule{0pt}{9pt}
\!\!\!\! 
	$1$         & $0.0104015$ & $ 1.1019564 \cdot 10^{-2}$ & $13$  \\
	$2$         & $0.0176409$ & $9.5284061 \cdot 10^{-4}$ & $13$  \\
	$3$         & $0.0219541$ & $2.5263758 \cdot 10^{-4}$ & $14$  \\
	$4$         & $0.0243116$ & $6.5050153 \cdot 10^{-5}$ & $13$  \\
	$5$         & $0.0255439$ & $1.6503282 \cdot 10^{-6}$ & $13$  \\
	$6$         & $0.0261739$ & $4.1561289 \cdot 10^{-6}$ & $13$  \\
	$7$         & $0.0264923$ & $1.0428313 \cdot 10^{-6}$ & $13$  \\
	$8$         & $0.0266524$ & $2.6118300 \cdot 10^{-7}$ & $13$  \\
	$9$         & $0.0267327$ & $6.5355110 \cdot 10^{-8}$ & $13$  \\ 
  $10$        & $0.0267728$ & $9.6346192 \cdot 10^{-9}$ & $13$  \\
	$11$        & $0.0267930$ & $1.7293467 \cdot 10^{-10}$ & $4$  \\
 \hline
\end{tabular}
\vspace{2mm}
\caption{Distance to singularity for the ORANI678 matrix: 
computed values $\eps_k$, $\phi(\eps_k) =  |\lambda_{\min} (A + \eps_k E_k)|^2$ 
and number of eigenvalue computations of the inner rank-$1$ algorithm.\label{tab:ORdsing}}
\end{center}
\end{table}


The average CPU time of an outer iteration is around $528.6$ seconds, which is due to the 
augmented computational cost  required by the routine eigs. The average number of eigentriplets 
evaluation is $n_{\rm eig} = 12$.

\subsection{The matrix FIDAPM11 from the SPARSKIT collection}

The matrix $A$ is now a sparse real unsymmetric square matrix  taken from the set ECONAUS. It has dimension $n = 22294$ and a number
of nonzero entries $nz = 623554 \approx 30 n$.  Its sparsity pattern is plotted in Figure \ref{fig:Orani1}. 

We have set $\eps=0.5$ and applied our algorithms to the minimization problem \eqref{eq:optimizS} with
$
f(\lambda,\clambda) = -\lambda \clambda = - |\lambda|^2 
$ 
and $\cS$ the space of real matrices with the sparsity pattern of $A$,
and the target eigenvalue is the one with largest real part. We are thus aiming to compute the \emph{structured $\eps$-pseudospectral radius} of $A$.

Integrating both ODEs \eqref{ode-E-S} and \eqref{ode-E-S-1}, we obtain the same optimizer
$\lambda = 1.9716893.$ 
The number of computed eigen-triplets is $n_{\rm eig} = 107$ and $n_{\rm eig} = 99$, with a slight advantage of the rank-$1$
method. The CPU time is close to $32.95$ and $31.32$ seconds respectively.
Also in this case $u$ and $v$ turn out to be real so that the gain in terms of memory requirements is significant,  
$623554/44588 \approx 13.98$. 

\begin{figure}[ht] \label{fig:Ftk}
\vspace{-2cm}
\centerline{
\includegraphics[trim=0 -100 -20 -20,clip,scale=0.33]{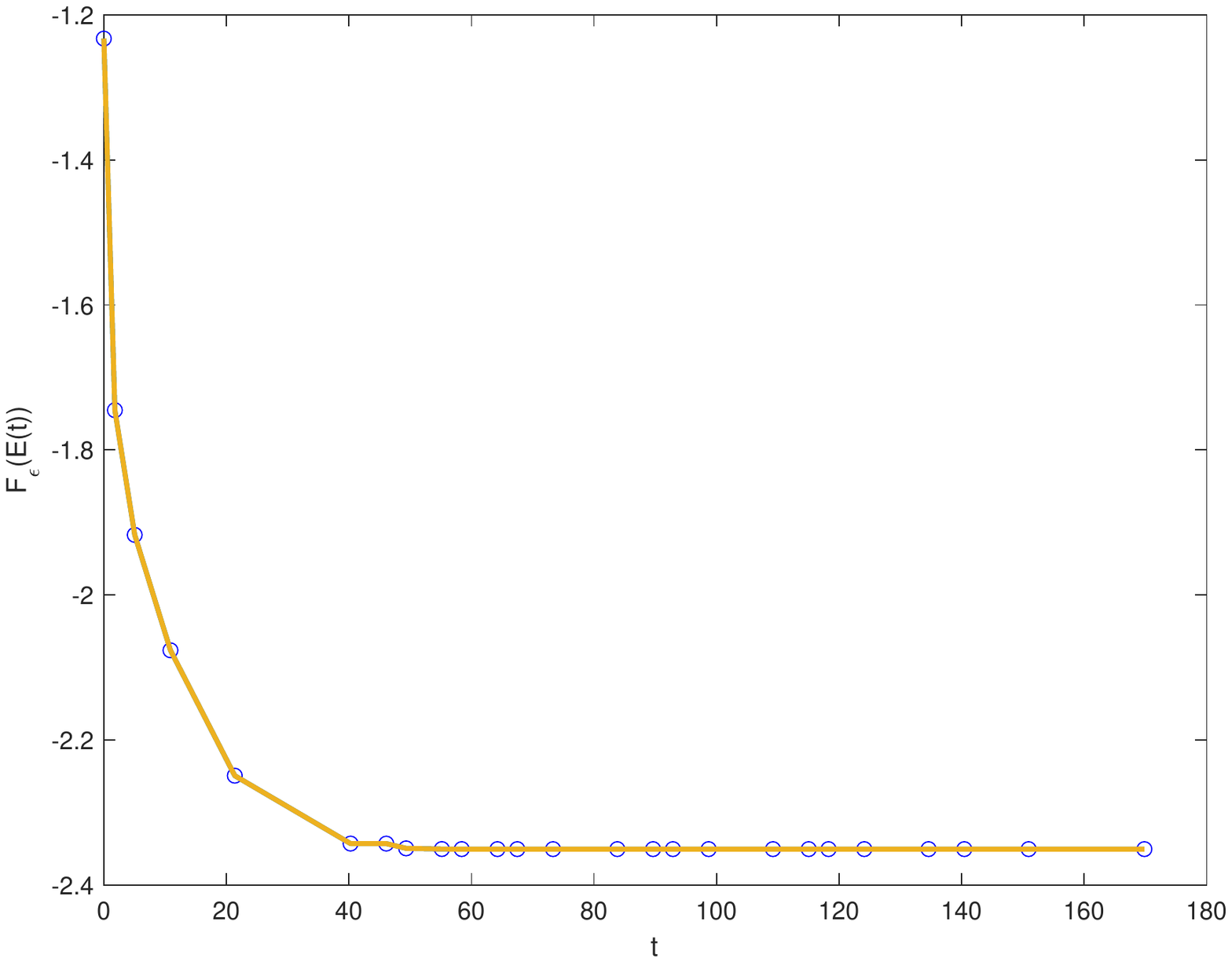}
\hspace{-15mm} 
\includegraphics[trim=0 -100 -20 -20,clip,scale=0.33]{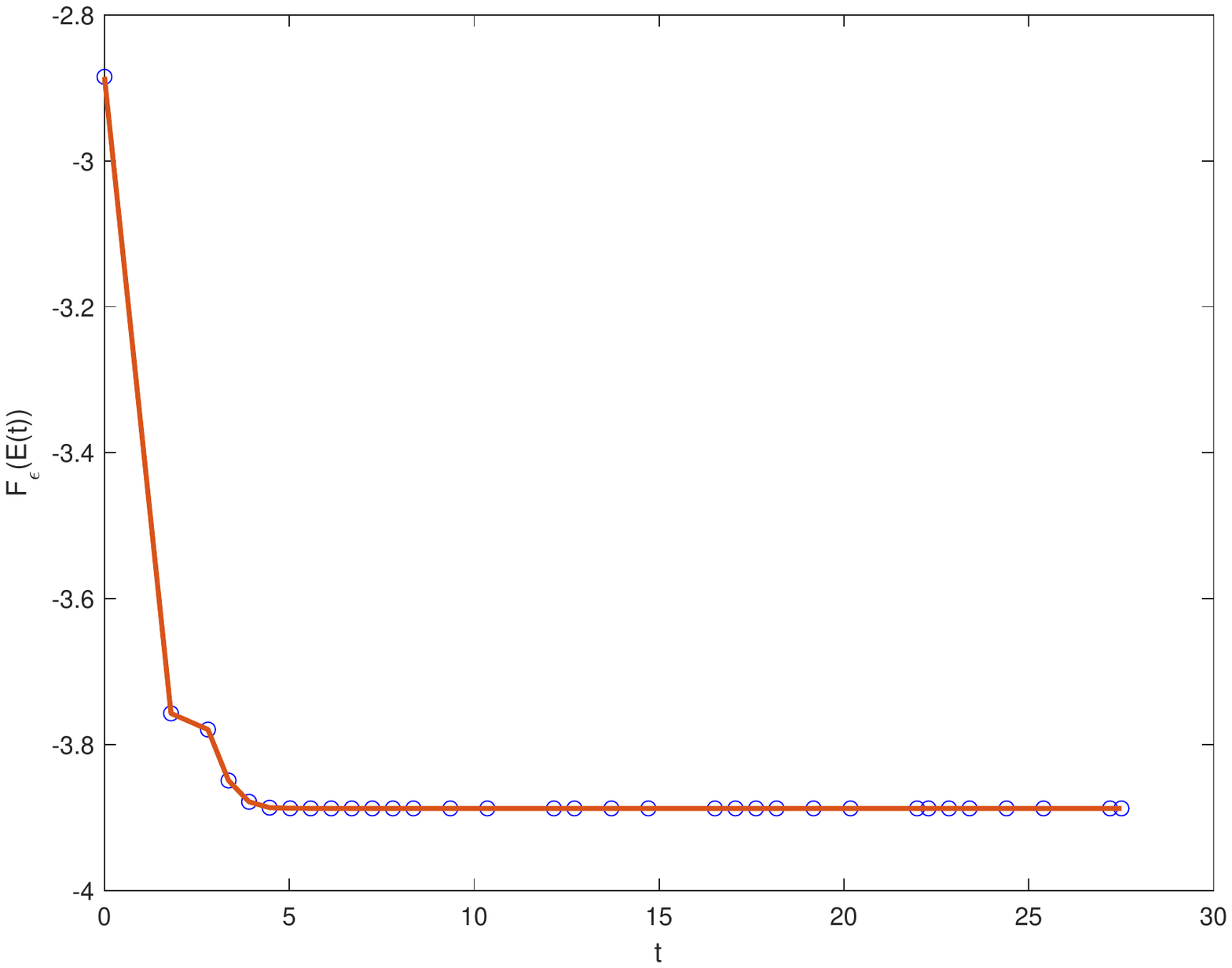}
}
\vspace{-3.6cm}
\caption{Behavior of the functional $F_\eps \left(E(t_k)\right)$ 
in the numerical integration by Algorithm {\rm \ref{alg_prEul}} 
for the matrix ORANI678 with $f(\lambda,\clambda) = - \Re(\lambda)$ (left picture) and for the matrix FIDAPM11
with $f(\lambda,\clambda) = -|\lambda|^2$ (right picture). In both cases $F_\eps \left(E(t_k)\right)$
decays monotonically with $k$. }
\end{figure}

\subsection{An example of control of the Stokes problem}
 
We consider an example from \cite{HMM19}, which arises in the discretization of 
the $2$-dimensional Stokes problem on a uniform quadratic grid.
Setting $25$ grid points on both sides of the square, we get a sparse matrix $A$ 
($J-R$ in the notation of \cite{HMM19}) which has dimension $n=1824$, while we choose
the control matrices $B$ and $C=B^\top$ to have size $n \times k$ and $l \times n$, 
respectively with $k=l=40$, randomly i.i.d. entries and unit Frobenius norm.

The matrix $A$ has the rightmost eigenvalue $\lambda = -6.4343098 \cdot 10^{-4}$,
which suggests a non-robust Hurwitz stability.

Running our algorithm on this example, we find the structured stability radius to be $0.0384039$, 
which is 60 times larger than $|\lambda|$.

Since the matrix is sparse we can exploit favorably the matrix vector products of the form 
(with $p \in \R^n$ the vector, $Z = u v^* \in \cM_1$, and $\rho$ the normalization factor)
\[
\left(A + \eps \rho B\,B^\dagger Z C^\dagger C \right)\,p,
\]
whose cost is linear in $n$.  

In Table \ref{tab:STradius} we show the Newton iteration where the number of eigentriplets evaluation is again
indicated by $n_{\rm eig}$.  The quadratically convergent behavior is evident.
\begin{table}[hbt]
\begin{center}
\begin{tabular}{|l|l|l|l|}\hline
  $k$ & $\eps_k$ & $\phi(\eps_k)$ & $\#$ eigs \\
 \hline
\rule{0pt}{9pt}
\!\!\!\! 
  $0$         & $0$ & $-6.4343098 \cdot 10^{-4}$ & $1$  \\
	$1$         & $0.02$ & $-3.1062242 \cdot 10^{-4} $ & $23$  \\
	$2$         & $0.0385299$ & $2.1414201 \cdot 10^{-6}$ & $26$  \\
	$3$         & $0.0384039$ & $9.7779625 \cdot 10^{-11}$ & $18$  \\
 \hline
\end{tabular}
\vspace{2mm}
\caption{Iterates for computing the structured stability radius for the Stokes problem matrix with
range- and corange-constrained perturbations; the optimal perturbation size $\eps^*$ is where $\ophi(\eps^*)=0$.
\label{tab:STradius}}
\end{center}
\end{table}

\subsection*{Acknowledgments}



Nicola Guglielmi acknowledges that his research was supported by funds from the
Italian MUR (Ministero dell'Universit\`a e della Ricerca) within the PRIN-17
Project ``Discontinuous dynamical systems: theory, numerics and applications''.
He also acknowledges affiliation to INdAM Research group GNCS (Gruppo Nazionale 
di Calcolo Scientifico).

\end{document}